\pdfoutput=1
\RequirePackage[l2tabu, orthodox]{nag}
\documentclass{article}
      \title{On Thompson's group T and algebraic K-theory}
     \author{Ross Geoghegan and Marco Varisco}
       \date{}

\usepackage{amsmath}
\usepackage{amsthm}
\usepackage{amssymb}
\usepackage{mathtools}
\usepackage{enumitem}
\usepackage{etoolbox}
\usepackage[pagebackref, pdfinfo={
      Title={On Thompson's group T and algebraic K-theory},
     Author={Ross Geoghegan and Marco Varisco},
    Subject={MSC2010: 19B28, 57M07, 20F65}
}]{hyperref}
\usepackage[alphabetic, backrefs, msc-links, nobysame]{amsrefs}

\newcommand*{\hurl}  [2][www.]{\href{http://#1#2}{\nolinkurl{#2}}}
\newcommand*{\hemail}[1]{\href{mailto:#1}{\nolinkurl{#1}}}
\newcommand*{\DOI}   [1]{\href{http://dx.doi.org/#1}{\nolinkurl{#1}}}
\newcommand*{\arXiv} [1]{\href{http://www.arxiv.org/abs/#1}{\nolinkurl{arXiv:#1}}}

\setlist{labelindent=\parindent, leftmargin=*, itemsep=0pt}

\numberwithin{equation}{section}

\theoremstyle{plain}
  \newtheorem{conjecture} [equation]{Conjecture}
  \newtheorem{corollary}  [equation]{Corollary}
  \newtheorem{lemma}      [equation]{Lemma}
  \newtheorem{proposition}[equation]{Proposition}
  \newtheorem{theorem}    [equation]{Theorem}

\theoremstyle{definition}
  \newtheorem{example}    [equation]{Example}

\DeclareMathAlphabet{\matheurm}      {U}{eur}{m}{n}
    \SetMathAlphabet{\matheurm}{bold}{U}{eur}{b}{n}

\newcommand*{\define}[5]{%
  \ifstrequal{#2}{*}{\expandafter#1\expandafter*}{\expandafter#1}%
  \csname#4#5\endcsname{#3{#5}}
}
\forcsvlist{\define{\newcommand}{*}{\mathcal}{C}}{%
  A,B,C,D,E,F,G,H,I,J,K,L,M,N,O,P,Q,R,S,T,U,V,W,X,Y,Z,%
}
\forcsvlist{\define{\newcommand}{*}{\mathbb}{I}}{%
  A,B,C,D,E,F,G,H,I,  K,L,M,N,O,P,Q,R,S,T,U,V,W,X,Y,Z%
}                

\forcsvlist{\define{\newcommand}{*}{\matheurm}{}}{Or, Sub}

\forcsvlist{\define{\DeclareMathOperator}{}{}{}}{coker, conhom, id, inn}

\forcsvlist{\define{\DeclareMathOperator}{*}{}{}}{colim}

\DeclareMathOperator*{\tensor}{\otimes}

\newcommand{\ts}{\textstyle}

\newcommand*{\TO}[1][]{\stackrel{#1}{\mathchoice{\longrightarrow}{\rightarrow}{\rightarrow}{\rightarrow}}}
\newcommand{\MOR}[4][]{#2\colon#3\TO[#1]#4}

\newcommand*{\rono}{\rho}       
\newcommand*{\Wh}{\mathit{Wh}}  
\newcommand*{\Fin}{{\mathcal{F}\mathrm{in}}}  
\newcommand*{\FinCyc}{{\mathcal{F}\mathrm{in}\mathcal{C}\mathrm{yc}}}

\begin{document}

\maketitle

\begin{abstract}
Using a theorem of L\"uck-Reich-Rognes-Varisco, we show that the Whitehead group of Thompson's group~$T$ is infinitely generated, even when tensored with the rationals.
To this end we describe the structure of the centralizers and normalizers of the finite cyclic subgroups of~$T$, via a direct geometric approach based on rotation numbers.
This also leads to an explicit computation of the source of the Farrell-Jones assembly map for the rationalized higher algebraic $K$-theory of the integral group ring of~$T$.
\end{abstract}


\section{Introduction and statement of results}\label{sec1}

Thompson's groups $F$ and $T$ are well-known groups having both type $F_\infty$ and infinite cohomological dimension.
Recall that $F$ and~$T$ can be defined as the groups of orientation-preserving dyadic piecewise-linear homeomorphisms of the closed unit interval $I=[0,1]$ and of the circle~$S^1=\IR/\IZ$; see Section~\ref{sec:T} (and \cite{CFP} for a comprehensive introduction).

Essentially nothing has been known about the algebraic $K$-theory of these groups.
Here we show that the Whitehead group of~$T$ is infinitely generated, even when tensored with the rationals.
More precisely, our main theorem is the following.

\begin{theorem}
\label{thm:K}
The Farrell-Jones assembly map in algebraic $K$-theory induces an injective homomorphism
\begin{equation}
\label{eq:main}
\colim_{k\in\IN}\Wh(C_k)\tensor_\IZ\IQ
\TO
\Wh(T)\tensor_\IZ\IQ
\end{equation}
and in particular $\Wh(T)\tensor_\IZ\IQ$ is an infinite dimensional $\IQ$-vector space.
\end{theorem}

On the left-hand side of~\eqref{eq:main} $C_k=\IZ/k\IZ$ denotes the finite cyclic group of order~$k$.
It is well known that $\Wh(C_k)\tensor_\IZ\IQ\neq0$ for all~$k\not\in\{1,2,3,4,6\}$; see for example~\cite{Oliver}*{top of page~6}.
The colimit in~\eqref{eq:main} is taken over the poset~$\IN$ with respect to the divisibility relation, and the homomorphisms $\Wh(C_k)\TO\Wh(C_\ell)$ induced by $C_k=\IZ/k\IZ\TO C_\ell=\IZ/\ell\IZ$, $1\mapsto\frac{\ell}{k}$ whenever $k\mid\ell$.
The~map in~\eqref{eq:main} is induced by identifying~$C_k$ with the cyclic subgroup~$\langle\gamma_k\rangle$ of~$T$ generated by the pseudo-rotation of order~$k$ from Example~\ref{ex:pseudo-rot}; see the proof of Corollary~\ref{cor:K}.

Theorem~\ref{thm:K} is a direct application to~$T$ of the paper~\cite{LRRV}.
That work and its applicability here are discussed in Section~\ref{sec:K}, where we also obtain results about the higher algebraic $K$-theory groups of the integral group ring of~$T$.
The ingredients about~$T$ needed for this application are summarized in the following theorem.

\begin{theorem}
\label{thm:T}
Every finite subgroup of~$T$ is cyclic, and for every integer $k\geq0$ there is exactly one conjugacy class in~$T$ of cyclic subgroups of order~$k$.
Moreover, for every finite cyclic subgroup~$C$ of~$T$, the centralizer~$Z_TC$ and the normalizer~$N_TC$ of~$C$ in~$T$ are equal, and there is a short exact sequence
\begin{equation}
\label{eq:T}
1\TO C\TO Z_TC\TO T\TO 1
\,.
\end{equation}
\end{theorem}

We proved Theorem~\ref{thm:T} in 2007 and only lately became aware that essentially the same result (but without the observation about normalizers, which is important for our application) appeared in Matucci's 2008 thesis \cite{Matucci}*{Theorem~7.1.5}, and was subsequently generalized in~\citelist{\cite{MPN} \cite{MPMN}}.
The full details of our proof of Theorem~\ref{thm:T} are given in Section~\ref{sec:T}.

The group~$T$ is of type~$F_\infty$ by a theorem of Brown and Geoghegan; see for example~\cite{Brown}*{Remark~2 on page~56}, where this is shown to follow immediately from the same result for the group~$F$~\cite{BG}.
Thus the short exact sequence~\eqref{eq:T} of Theorem~\ref{thm:T} has the following corollary (see for example~\cite{book}*{Section~7.2}).

\begin{corollary}
\label{cor:T}
For every finite cyclic subgroup~$C$ of~$T$ the centralizer~$Z_T C$ of~$C$ in~$T$ is of type~$F_\infty$.
Moreover, for every~$s\in\IN$, the abelian group $H_s(BZ_T C;\IZ)$ is finitely generated, and $H_s(BZ_T C;\IQ)\cong H_s(BT;\IQ)$.
\end{corollary}

The rational homology groups~$H_s(BT;\IQ)$ of~$T$ (and hence also of~$Z_T C$) are completely known thanks to a theorem of Ghys and Sergiescu.
In fact, in~\cite{Ghys-Sergiescu}*{Corollaire~C on pages 187--188} it is proved that $H^1(T;\IZ)=0$, $H^2(T;\IZ)\cong\IZ\oplus\IZ$ with natural generators $\alpha$ and~$\chi$, and $H^*(T;\IQ)\cong\IQ[\alpha,\chi]$.


\sloppy
\section{Thompson's group~T and centralizers of \mbox{finite}
 subgroups}
\fussy
\label{sec:T}

In this section we recall the definition of Thompson's groups $F$ and~$T$, and then prove Theorem~\ref{thm:T};
see Theorem~\ref{thm:centralizers} and Corollary~\ref{cor:conj} below.

We say that an interval $I\subset\IR$ is \emph{dyadic} if its endpoints are dyadic rationals.
If $I$ and~$J$ are closed dyadic intervals, a homeomorphism $\MOR{f}{I}{J}$ is called \emph{dyadic piecewise linear}, or \emph{DPL} for short, if $f$ is piecewise linear, the breakpoints occur at dyadic rational points, and the slopes are integer powers of~$2$.
Notice that the inverse of a DPL homeomorphism is again DPL.
\emph{Thomspon's group~F} is defined as the group of orientation-preserving DPL homeomorphisms of~$[0,1]$.

We define an \emph{$\IR$-space} to be a pair $(X,p)$ where $X$ is a topological space and $\MOR{p}{\IR}{X}$ is a covering map.
In other words, $X$ is a connected $1$-dimensional manifold together with a chosen universal covering map~$p$.
We consider every $\IR$-space to be oriented via~$p$.
The primary example is of course $X=S^1=\IR/\IZ$ together with the usual universal covering map~$\MOR{u}{\IR}{S^1}$.

Let $(X,p)$ and $(Y,q)$ be $\IR$-spaces, and let $\MOR{f}{X}{Y}$ be a map.
We say that $f$ is \emph{locally DPL} (short for local dyadic piecewise linear homeomorphism) if for every~$x\in X$ there exist closed dyadic intervals $I$, $J$ in~$\IR$ such that:
\begin{itemize}
\item $p_{|I}$ and $q_{|J}$ are embeddings;
\item $x$ belongs to the interior of~$p(I)$ and $f(x)$ belongs to the interior of~$q(J)$;
\item $f$ induces a homeomorphism~$\MOR{f_{|p(I)}}{p(I)}{q(J)}$;
\item and the composition
\[
I\TO[p_{|}] p(I) \TO[f_{|}] q(J) \TO[q_{|}^{-1}] J
\]
is a DPL homeomorphism.
\end{itemize}

If $(X,p)$ is an~$\IR$-space, then we define $H(X,p)=H(X)$ to be the group of all orientation-preserving homeomorphisms of~$X$, and~$T(X,p)$ to be the subgroup of~$H(X,p)$ consisting of those orientation-preserving homeomorphisms that are locally DPL.
\emph{Thompson's group~$T$} is defined as~$T=T(S^1,u)$.
Similarly we write~$H=H(S^1,u)$.
Thompson's group~$F$ can then be identified with the subgroup of~$T$ fixing a base point.

\begin{example}[pseudo-rotations]
\label{ex:pseudo-rot}
The following elements of~$T$ play an important role in our work.
Given $q\geq2$ we denote by $\gamma_q\in T$ be the pseudo-rotation of order~$q$, i.e., the element of~$T$ (called~$C_{q-2}$ in~\cite{CFP}*{pages~236--237}) that cyclically permutes the images of the $q$ intervals
\begin{equation}
\label{eq:gamma}
\ts
\left[0,\frac{1}{2}\right],
\left[\frac{1}{2},\frac{3}{4}\right],
\ldots,
\left[1-\frac{1}{2^{j}},1-\frac{1}{2^{j+1}}\right],
\ldots,
\left[1-\frac{1}{2^{q-1}},1\right]
\end{equation}
and is affine on each of them.
\end{example}

The main result in this section is the following.

\begin{theorem}
\label{thm:centralizers}
Let $C$ be a finite subgroup of~$H$.
Then $C$ is cyclic, the centralizer~$Z_H C$ and the normalizer~$N_H C$ of~$C$ in~$H$ are equal, and there is a short exact sequence
\[
1\TO C\TO Z_H C \TO H \TO1
\,.
\]
Moreover, if $C<T$, then $Z_T C=N_T C$, and there is a short exact sequence
\[
1\TO C\TO Z_T C \TO T \TO1
\,.
\]
\end{theorem}

The proof of Theorem~\ref{thm:centralizers} uses Poincar\'e rotation numbers.
We now recall their definition and basic properties, and we refer to~\cite{KH}*{Chapter~11} for more details and proofs.

Given $h\in H$, choose a lift $\MOR{\overline{h}}{\IR}{\IR}$ such that $u\overline{h}=hu$, and choose a point~$x\in\IR$.
Define
\[
\rono(h)=\IZ + \lim_{n\to\infty}\frac{\overline{h}^n(x)-x}{n} \in\IR/\IZ
\,
.
\]
Then $\rono(h)\in\IR/\IZ$ is independent of the choices of~$\overline{h}$ and~$x$ (see~\cite{KH}*{Proposition~11.1.1}), and it is called the \emph{rotation number} of~$h$.

\begin{proposition}
\label{prop:KH}
Let $h,g\in H$ and let $m$ be an integer.
\begin{enumerate}[label=(\roman*)]
\item\label{KH-rotation}
If $h(x)=x+\vartheta\mod\IZ$, i.e., if $h$ is a rotation by~$2\pi\vartheta$, then $\rono(h)=\vartheta$.
In particular, $\rono(\id_{S^1})=0$.
\item\label{KH-power}
$\rono(h^m)=m\rono(h)$.
\item\label{KH-conj}
$\rono(hgh^{-1})=\rono(g)$.
\item\label{KH-fixed}
If $\rono(h)=0$, then $h$ has a fixed point.
\item\label{KH-torsion}
If $h\neq\id_{S^1}$ has finite order, then $\rono(h)\in\IQ/\IZ$ and $\rono(h)\neq0$.
Let $\rono(h)=\frac{p}{q}$ with $(p,q)=1$ and~$0<p<q$.
Then the order of~$h$ is~$q$;
for every $x\in S^1$ the ordering of
$\bigl\{x, h(x), h^2(x), \ldots, h^{q-1}(x)\bigr\}$
in~$S^1$ is the same as that of
$\bigl\{0, \frac{p}{q}, \frac{2p}{q}, \frac{(q-1)p}{q}\bigr\}$;
and $h$ is conjugate to the rotation by~$2\pi\frac{p}{q}$.
\end{enumerate}
\end{proposition}

\begin{proof}
Statements \ref{KH-rotation} and~\ref{KH-power} follow immediately from  the definition, whereas \ref{KH-conj} and~\ref{KH-fixed} are proved in~\cite{KH}*{Propositions 11.1.3 and~11.1.4}.

\ref{KH-torsion}
Let $h\neq\id_{S^1}$ have finite order.
From~\cite{KH}*{Proposition~11.1.1} we have that $\rono(h)\in\IQ/\IZ$.
From~\ref{KH-fixed} and Lemma~\ref{lem:torsion} below we conclude that $\rono(h)\neq0$.
So let $\rono(h)=\frac{p}{q}$ with $(p,q)=1$ and~$0<p<q$.
Suppose that the order of~$h$ is~$m$.
Then, using \ref{KH-rotation} and~\ref{KH-power}, $0=\rono(\id_{S^1})=\rono(h^m)=m\rono(h)=m\frac{p}{q}$, and therefore $q|m$ since $(p,q)=1$.
On the other hand $\rono(h^q)=q\rono(h)=q\frac{p}{q}=0\in\IR/\IZ$, and therefore from~\ref{KH-fixed} and Lemma~\ref{lem:torsion} we conclude that $h^m=\id_{S^1}$ and hence $m|q$.
So the order of~$h$ is~$q$.
The last statements then follow from~\cite{KH}*{Proposition~11.2.1}.
\end{proof}

\begin{lemma}
\label{lem:torsion}
If $h\in H$ has finite order and has a fixed point, then $h=\id_{S^1}$.
\end{lemma}

\begin{proof}
If~$h$ has a fixed point, then $h$ induces an orientation-preserving homeomorphism of a closed interval.
Since the group of orientation-preserving homeomorphisms of a closed interval is torsion-free, if $h$ also has finite order then~$h=\id_{S^1}$.
\end{proof}
\pagebreak

\begin{corollary}
\label{cor:conj}
Any two cyclic subgroups of~$H$ (respectively, of~$T$) with the same order are conjugate in~$H$ (respectively, in~$T$).
\end{corollary}

\begin{proof}
Let $C$ be a cyclic subgroup of~$H$ with order~$q$.
Proposition~\ref{prop:KH}\ref{KH-power} implies that $C$ has a unique generator~$g$ with rotation number~$\frac{1}{q}$, and \ref{KH-torsion} implies that $g$ is conjugate in~$H$ to the rotation by~$\smash{\frac{2\pi}{q}}$, and therefore the corollary is true for~$H$.
So assume that $g\in T$.
By Proposition~\ref{prop:KH}\ref{KH-torsion}, the ordering of
$\CO=\bigl\{0, g(0), g^2(0), \ldots, g^{q-1}(0)\bigr\}$
in~$S^1$ is the same as that of
$\bigl\{0, \frac{p}{q}, \frac{2p}{q}, \frac{(q-1)p}{q}\bigr\}$,
and so each $g^k(0)$ is a dyadic rational.
Think of~$\CO$ as a dyadic subdivision of~$S^1$.
Then there is a finer dyadic subdivision~$\CO'$ such that $g$ is affine on each segment of~$\CO'$.
Now let $\gamma_q\in T$ be the pseudo-rotation of order~$q$ from Example~\ref{ex:pseudo-rot}.
Define $\CO''$ to be the dyadic subdivision of~\eqref{eq:gamma} corresponding to~$\CO'$, so that $\gamma_q$ is also affine on each segment of~$\CO''$,
and define $h\in T$ to be the locally DPL homeomorphism that maps each segment of~$\CO'$ affinely onto the corresponding segment of~$\CO''$.
Then $hgh^{-1}=\gamma_q$, and therefore the corollary is also true for~$T$.
\end{proof}

We are now ready to prove Theorem~\ref{thm:centralizers}.

\begin{proof}[Proof of Theorem~\ref{thm:centralizers}]
Let $C$ be a finite subgroup of~$H$.
Assume that~$C\neq 1$, otherwise there is nothing to prove.
Define $S^1_0=C\backslash S^1$ to be the quotient, and denote by~$\MOR{q}{S^1}{S^1_0}$ the quotient map.

We first show that $C$ is cyclic.
By Lemma~\ref{lem:torsion}, if~$g\in C$ has a fixed point, then $g=\id_{S^1}$.
It follows that $\MOR{q}{S^1}{S^1_0}$ is a covering map and that $S^1_0$, being a closed $1$-dimensional manifold, is homeomorphic to~$S^1$, and therefore $C$ is cyclic by covering space theory.

Notice that $S^1_0$ together with the composition~$qu$ is an $\IR$-space.
We abbreviate $H_0=H(S^1_0,qu)$ and $T_0=T(S^1_0,qu)$.

Fix a generator~$g$ of~$C$.
By Proposition~\ref{prop:KH}\ref{KH-torsion}, we know that $\rono(g)=\frac{p}{q}\neq0$, with $(p,q)=1$ and $0<p<q$, and $q$ is the order of~$C$.
Let $s$ and $t$ be such that $sp+tq=1$.
Then $g^s(1)$ is the element in the orbit of~$1\in S^1$ coming directly after~$1$ in the cyclic order.
Let $\ell$ be the length in~$S^1$ of $[1,g^s(1)]$.
Then $S^1_0$ can be identified with $\IR/\ell\IZ$, and multiplication by $\frac{1}{\ell}$ induces a homeomorphism $\MOR{f}{S^1_0}{S^1}$.
It follows that conjugation by~$f$ yields an isomorphism $H_0\cong H$.
Moreover, if $g\in T$, then $\ell$ is a dyadic rational and therefore the homeomorphism~$f$ is locally DPL, so conjugation by~$f$ restricts to an isomorphism~$T_0\cong T$.

To prove that $N_H C=Z_H C$, let $h\in N_H C$ be given.
Then $hgh^{-1}=g^m$ for some integer~$m$.
By Proposition~\ref{prop:KH}\ref{KH-power}-\ref{KH-conj} we see that $\rono(g)=m\rono(g)$, and by Proposition~\ref{prop:KH}\ref{KH-fixed} and Lemma~\ref{lem:torsion} we see that $\rono(g)\neq0$.
Therefore $m=1$ and so $h\in Z_H C$.
Since obviously $Z_H C\leq N_H C$, we conclude that~$N_H C=Z_H C$.

Since $Z_H C$ acts on the quotient~$S^1_0=C\backslash S^1$, we get a group homomorphism~$\MOR{\pi}{Z_H C}{H_0}$.
We are going to show next that there is a short exact sequence
\begin{equation}
\label{eq:ses-H}
1\TO C\TO Z_H C \TO[\pi] H_0 \TO1
\,,
\end{equation}
i.e., that $\ker\pi=C$ and that $\pi$ is surjective.
Since $H_0\cong H$, as observed above, this will prove the first part of the theorem.

To show that $\ker\pi=C$, let $h\in\ker\pi$ be given.
Then for any $x\in S^1$,\linebreak $h(x)=g^{m(x)}(x)$ for some integer~$m(x)$.
By continuity and since $S^1$ is connected, it follows that~$m(x)$ is constant, i.e., that $h\in C$.
Since obviously $C\leq\ker\pi$, we conclude that $\ker\pi=C$.

To show that $\pi$ is surjective, let $h_0\in H_0$ be given.
Choose a basepoint $x_0\in S^1_0$ and define $y_0=h_0(x_0)$.
Since $h_0$ is freely homotopic to~$\id_{S^1_0}$, choose such a homotopy and let $\alpha$ be the track of this homotopy at~$x_0$; $\alpha$ is then a path from $x_0$ to~$y_0$.
It follows that given any loop~$\omega$ at~$x_0$, $\omega$ is homotopic to~$\alpha\cdot(h_0\omega)\cdot\alpha^{-1}$ relative to~$x_0$.
Now choose $x\in S^1$ such that $q(x)=x_0$, and let $\widetilde{\alpha}$ be the lift of~$\alpha$ starting at~$x$.
Define $y=\widetilde{\alpha}(1)$ and $y_0=q(y)=\alpha(1)$.

We want to show that $h_0$ lifts to a homeomorphism $h\in H$ such that:
\begin{equation}
\label{eq:lift}
qh=h_0 q
\quad\text{and}\quad
h(x)=y
\,.
\end{equation}
We first show that $\MOR{h_0 q}{S^1}{S^1_0}$ lifts to a map $\MOR{h}{S^1}{S^1}$ satisfying~\eqref{eq:lift}.
It follows then easily that~$h\in H$.

By covering space theory, it is enough to show that if $\gamma$ is any loop in~$S^1$ at~$x$ then there is a loop~$\sigma$ in~$S^1$ at~$y$ such that $h_0q\gamma$ is homotopic to~$q\sigma$ relative to~$y_0$.
Given~$\gamma$, since $q\gamma$ is homotopic to~$\alpha\cdot(h_0q\gamma)\cdot\alpha^{-1}$ relative to~$x_0$ and $q\gamma$ lifts to a loop at~$x$, it follows that $\alpha\cdot(h_0q\gamma)\cdot\alpha^{-1}$ lifts to a loop~$\tau$ at~$x$.
Let $\sigma$ be the lift of~$h_0q\gamma$ at~$y$.
We claim that $\sigma$ is a loop.
Indeed, if $\sigma(1)=g^m y$ for some integer~$m$, then $\tau(1)=g^m x$.
Hence $m=0$ and $\sigma$ is a loop, as claimed.
Therefore $h_0$ lifts to an~$h\in H$ satisfying~\eqref{eq:lift}.

It only remains to show that $h$ commutes with~$g$, i.e., $h\in Z_H(C)$.
Let $x\in S^1$.
Let $\rono(g)=\frac{p}{q}$ with $(p,q)=1$.
By Proposition~\ref{prop:KH}\ref{KH-torsion} we know that the cyclic order of
$Cx=\bigl\{x, gx, g^2x, \ldots, g^{q-1}x\bigr\}$
in~$S^1$ is the same as that of
$\bigl\{0, \frac{p}{q}, \frac{2p}{q}, \frac{(q-1)p}{q}\bigr\}$.
Since $h$ preserves cyclic order, it follows that the cyclic orders of
\begin{align*}
Ch(x)=\,&\bigl\{h(x), gh(x), g^2h(x), \ldots, g^{q-1}h(x)\bigr\}\,,
\\
h\bigl(Cx\bigr)=\,&\bigl\{h(x), h(gx), h(g^2x), \ldots, h(g^{q-1}x)\bigr\}\,,
\\
\text{and}\quad
&\bigl\{\ts0, \frac{p}{q}, \frac{2p}{q}, \frac{(q-1)p}{q}\bigr\}
\end{align*}
are all the same.
Since $qh=h_0q$, $h$ sends orbits to orbits, and therefore we have $Ch(x)=h\bigl(Cx\bigr)$, from which it follows that $gh(x)=h(gx)$.

Finally, assume that $C<T$.
Since $Z_T C=T\cap Z_H C$ and $N_T C=T\cap N_H C$, it is now clear that $Z_T C=N_T C$.
For all $h\in H$, since membership in~$T(X,p)$ is a local property, $h\in T=T(S^1,u)$ if and only if~$\pi(h)\in T_0=T(S^1_0,qu)$, therefore \eqref{eq:ses-H} induces a short exact sequence
\begin{equation}
\label{eq:ses-T}
1\TO C\TO Z_T C \TO[\pi_{|}] T_0 \TO1
\,.
\end{equation}
But as observed above, $T_0\cong T$, and so the theorem is proved.
\end{proof}


\section{Assembly maps and algebraic K-theory of~T}
\label{sec:K}

In this last section we review assembly maps and isomorphism conjectures in algebraic $K$-theory, focusing on the rationalized case and referring the reader to~\citelist{\cite{LR} \cite{Lueck-ICM}} for comprehensive surveys.
Then we explain the main results of~\cite{LRRV} and how they imply Theorem~\ref{thm:K} as well as a generalization to higher algebraic $K$-theory.

Let $G$ be a discrete group.
The algebraic $K$-theory groups~$K_n(\IZ G)$ of the integral group ring of~$G$ play a central role in geometric topology, in particular in the classification of high-dimensional manifolds and their automorphisms.
Arguably the most important $K$-theoretic invariant is the \emph{Whitehead group}~$\Wh(G)$, which classifies high-dimensional $h$-cobordisms, and which is defined as the quotient of~$K_1(\IZ G)=\bigl(\bigcup_{k\in\IN}GL_k(\IZ G)\bigr)_{\mathrm{ab}}$ by the image of the $1$-by-$1$ invertible matrices~$(\pm g)$, $g\in G$.
The following conjecture is one of the most well-known and consequential open problems in this area.

\begin{conjecture}
\label{conj:Wh}
If $G$ is torsion-free, then $\Wh(G)=0$.
If $G$ has torsion, then the inclusions of finite subgroups~$H$ of~$G$ induce an injective homomorphism
\begin{equation}
\label{eq:Wh}
\colim_{H\in\Sub_\Fin G}\Wh(H)\tensor_\IZ\IQ\TO\Wh(G)\tensor_\IZ\IQ
\,.
\end{equation}
\end{conjecture}

The colimit in~\eqref{eq:Wh} is taken over the finite subgroup category~$\Sub_\Fin G$, whose objects are the finite subgroups $H$ of~$G$ and whose morphisms are defined as follows.
Given subgroups $H$ and~$K$ of~$G$, let $\conhom_G(H,K)$ be the set all group homomorphisms $H\TO K$ given by conjugation by an element of~$G$.
The group $\inn(K)$ of inner automorphisms of~$K$ acts on~$\conhom_G(H,K)$ on the left by post-composition.
The set of morphisms in~$\Sub_\Fin G$ from $H$ to~$K$ is then defined as the quotient~$\inn(K)\backslash\conhom_G(H,K)$.
For example, in the special case when $G$ is abelian, then $\Sub_\Fin G$ is just the poset of finite subgroups of~$G$ ordered by inclusion.
Equivalently, the colimit in~\eqref{eq:Wh} could be taken over the restricted orbit category~$\Or_\Fin G$, which has as objects the homogeneous $G$-sets~$G/H$ for any finite subgroup~$H$ of~$G$, and as morphisms the $G$-equivariant maps.
The relation between $\Sub_\Fin G$ and $\Or_\Fin G$ and the equivalence of the two approaches is explained, for example, in~\cite{LRV}*{page~152, Lemma~3.11}.

Conjecture~\ref{conj:Wh} is known to be true for all Gromov hyperbolic groups \cite{BLR} and all CAT(0)-groups~\cite{BL}, for example.
One of the most interesting open cases of Conjecture~\ref{conj:Wh} is Thompson's group~$F$: is~$\Wh(F)=0$?
Our main result, Theorem~\ref{thm:K}, is that for Thompson's group~$T$ Conjecture~\ref{conj:Wh} is true.

Before explaining this, we want to discuss how Conjecture~\ref{conj:Wh} is a special case of the more general Farrell-Jones Conjecture in algebraic $K$-theory.
This conjecture asserts that certain \emph{assembly maps} are isomorphisms.
The targets of the assembly maps are the algebraic $K$-theory groups~$K_n(\IZ G)$ that we are interested in.
The sources are other groups that are easier to compute and homological in nature, and that only depend on the algebraic $K$-theory of relatively ``small'' subgroups of~$G$.
The construction of these assembly maps is rather technical, and we will not explain it here---see e.g.~\cites{LR, LRRV} for details.
But the picture simplifies after rationalizing, i.e., after tensoring with~$\IQ$, and we are going to focus on it now.

The \emph{rationalized classical assembly map} for~$K_n(\IZ G)$, $n\in\IZ$, is a homomorphism
\begin{equation}
\label{eq:classical}
\bigoplus_{\substack{s,t\geq0\\s+t=n}}
H_s(BG;\IQ)\tensor_\IQ(K_t(\IZ)\tensor_\IZ\IQ)
\TO
K_n(\IZ G)\tensor_\IZ\IQ
\,.
\end{equation}
The \emph{rationalized Farrell-Jones assembly map} for~$K_n(\IZ G)$, $n\in\IZ$, is a homomorphism
\begin{equation}
\label{eq:FJ}
\hspace{-1pt}\bigoplus_{C\in(\FinCyc)}
\bigoplus_{\substack{s\geq0, t\geq-1\\s+t=n}}
H_s(BZ_GC;\IQ)\tensor_{\IQ[W_G C]}\Theta_C(K_t(\IZ C)\tensor_\IZ\IQ)
\TO
K_n(\IZ G)\tensor_\IZ\IQ
\,.
\end{equation}
Here $(\FinCyc)$ denotes the set of conjugacy classes of finite cyclic subgroups in~$G$, $Z_G C$ denotes the centralizer in~$G$ of~$C$, $W_G C$ denotes the quotient $N_G C/Z_G C$ of the normalizer modulo the centralizer, and $\Theta_C(K_t(\IZ C)\tensor_\IZ\IQ)$ is a direct summand of~$K_t(\IZ C)\tensor_\IZ\IQ$ naturally isomorphic to
\begin{equation*}
\coker\left(
\bigoplus_{D\lneqq C} K_t(\IZ D)\tensor_\IZ\IQ
\TO K_t(\IZ C)\tensor_\IZ\IQ
\right)
\!.
\end{equation*}
The dimensions of the $\IQ$-vector spaces $\Theta_C(K_n(\IZ C)\tensor_\IZ\IQ)$ can be explicitly computed; see~\cite{Patronas}*{Theorem on page~9}.

Moreover, the summand in the source of~\eqref{eq:FJ} corresponding to~$C=1$ is the same as the source of~\eqref{eq:classical}, since $K_{-1}(\IZ)=0$.
Therefore, if $G$ is torsion-free, then the classical and the Farrell-Jones assembly maps are the same.

\begin{conjecture}[Rationalized Farrell-Jones Conjecture]
\label{conj:FJ}
For any group~$G$ and for any~$n\in\IZ$, the rationalized Farrell-Jones assembly map~\eqref{eq:FJ} is an isomorphism.
In particular, if $G$ is torsion-free, then the map~\eqref{eq:classical} is an isomorphism for any~$n\in\IZ$.
\end{conjecture}

Conjecture~\ref{conj:FJ}, even in its much stronger integral version that we are not discussing here, is known to be true for all Gromov hyperbolic groups \cite{BLR} and all CAT(0)-groups~\cites{BL, Wegner}, for example.

Theorem~\ref{thm:T} and Corollary~\ref{cor:T} have the following immediate consequence.

\begin{corollary}
\label{cor:FJT}
The source of the rationalized Farrell-Jones assembly map for Thompson's group~$T$ is isomorphic to
\begin{equation}
\label{eq:FJT}
\bigoplus_{k\geq0}
\bigoplus_{\substack{s\geq0, t\geq-1\\s+t=n}}
H_s(BT;\IQ)\tensor_{\IQ}\Theta_{C_k}(K_t(\IZ C_k)\tensor_\IZ\IQ)
\end{equation}
for any~$n\in\IZ$.
In particular, if the Farrell-Jones conjecture is true for~$T$, then $K_n(\IZ T)\tensor_\IZ\IQ$ is isomorphic to~\eqref{eq:FJT} for any~$n\in\IZ$.
\end{corollary}

As we already remarked, thanks to theorems of Ghys-Sergiescu and Patronas, the dimension over~$\IQ$ of each individual summand in~\eqref{eq:FJT} is explicitly computable.

Now we recall a famous result about the injectivity of the rationalized classical assembly map.

\begin{theorem}[B\"okstedt-Hsiang-Madsen~\cite{BHM}]
\label{thm:BHM}
Let $G$ be any group, not necessarily torsion-free.
Assume that for every~$s\in\IN$ the abelian group $H_s(BG;\IZ)$ is finitely generated.
Then for every~$n\in\IN$ the rationalized classical assembly map~\eqref{eq:classical} is injective.
\end{theorem}

In particular, Theorem~\ref{thm:BHM} applies to Thompson's group~$F$, since any group of type~$F_\infty$ satisfies the assumption above.
However, this injectivity result produces no information about~$\Wh(G)$.

In~\cite{LRRV}, Theorem~\ref{thm:BHM} is generalized to the Farrell-Jones assembly map, yielding also information about~$\Wh(G)$.

\begin{theorem}[\cite{LRRV}*{Main Theorem~1.13}]
\label{thm:LRRV}
Let $G$ be any group.
Assume that for every finite cyclic subgroup $C$ of~$G$ the following conditions hold:
\begin{enumerate}[label=(\roman*)]
\item\label{LRRV-finiteness}
 for every~$s\in\IN$ the abelian group~$H_s(BZ_GC;\IZ)$ is finitely generated;
\item\label{LRRV-Schneider}
let $k$ be the order of~$C$ and let $\zeta_k$ be any primitive $k$\textsuperscript{th} root of unity; for every~$t\in\IN$ the natural homomorphism
\[
K_t(\IZ[\zeta_k])\longrightarrow
\prod_{p\in\IP}K_t\Bigl(\IZ_p\tensor_\IZ\IZ[\zeta_k];\IZ_p\Bigr)
\]
is injective after tensoring with~$\IQ$, where $\IP$ denotes the set of all primes and $\IZ_p$ denotes the ring of $p$-adic integers for~$p\in\IP$.
\end{enumerate}
Then the restriction of the rationalized Farrell-Jones assembly map~\eqref{eq:FJ} to the summands where~$t\neq-1$ induces an injective homomorphism
\begin{equation*}
\bigoplus_{C\in(\FinCyc)}
\bigoplus_{\substack{s,t\geq0\\s+t=n}}
H_s(BZ_GC;\IQ)\tensor_{\IQ[W_G C]}\Theta_C(K_t(\IZ C)\tensor_\IZ\IQ)
\TO
K_n(\IZ G)\tensor_\IZ\IQ
\end{equation*}
for every~$n\geq0$.
\end{theorem}

\begin{corollary}[\cite{LRRV}*{Theorem~1.1}]
\label{cor:LRRV}
Assume that a group~$G$ satisfies assumption~\ref{LRRV-finiteness} of Theorem~\ref{thm:LRRV}.
Then there is an injective homomorphism
\[
\colim_{H\in\Sub_\Fin G}\Wh(H)\tensor_\IZ\IQ
\TO
\Wh(G)\tensor_\IZ\IQ
\,,
\]
i.e., the second part of~Conjecture~\ref{conj:Wh} is true for~$G$.
\end{corollary}

Some remarks are in order about assumption~\ref{LRRV-Schneider} of Theorem~\ref{thm:LRRV}.
First of all, \ref{LRRV-Schneider} is true for all~$k$ if~$t=0,1$, and for all~$t$ if~$k=1$.
This explains why the assumption is absent from Theorem~\ref{thm:BHM} and Corollary~\ref{cor:LRRV}.
Moreover, assumption~\ref{LRRV-Schneider} is conjecturally always true, in the sense that it is automatically satisfied if a weak version of the Leopoldt-Schneider conjecture holds for cyclotomic fields; see~\cite{LRRV}*{Section~2} for details.

Now Theorem~\ref{thm:T} and its Corollaries \ref{cor:T} and~\ref{cor:FJT}, combined with Theorem~\ref{thm:LRRV} and Corollary~\ref{cor:LRRV}, immediately imply our main result; cf.~Theorem~\ref{thm:K}.

\begin{corollary}
\label{cor:K}
Conjecture~\ref{conj:Wh} is true for Thompson's group~$T$, and there is an injective homomorphism
\[
\colim_{k\in\IN}\Wh(C_k)\tensor_\IZ\IQ
\TO
\Wh(T)\tensor_\IZ\IQ
\,.
\]
In particular, $\Wh(T)\tensor_\IZ\IQ$ is an infinite dimensional $\IQ$-vector space.
Moreover, if assumption~\ref{LRRV-Schneider} of Theorem~\ref{thm:LRRV} holds for all~$k,t\in\IN$, then there is an injective homomorphism
\[
\bigoplus_{k\geq0}
\bigoplus_{\substack{s,t\geq0\\s+t=n}}
H_s(BT;\IQ)\tensor_{\IQ}\Theta_{C_k}(K_t(\IZ C_k)\tensor_\IZ\IQ)
\TO
K_n(\IZ T)\tensor_\IZ\IQ
\]
for all~$n\in\IN$.
\end{corollary}

\begin{proof}
The only step that remains to be explained is the identification
\begin{equation}
\label{eq:colim}
\colim_{k\in\IN}\Wh(C_k)
\cong
\colim_{H\in\Sub_\Fin T}\Wh(H)
\,,
\end{equation}
where on the left-hand side we have the colimit described right after Theorem~\ref{thm:K}.
Recall that all finite subgroups of~$T$ are cyclic.
Suppose that $C$ and~$D$ are finite subgroups of~$T$ of orders $k$ and~$\ell$, respectively, and assume that~$k\mid\ell$.
Then there is exactly one subgroup~$C'$ of~$D$ of order~$k$, and $C$ and~$C'$ are conjugate in~$T$ by Corollary~\ref{cor:conj}.
As explained in the proof of that Corollary, $C$ has a unique generator with rotation number~$\frac{1}{k}$, and the same is true for~$C'$.
Since rotation numbers are preserved by conjugation by Proposition~\ref{prop:KH}\ref{KH-conj}, we conclude that there is exactly one morphism in~$\Sub_\Fin T$ from~$C$ to~$D$.
Now, identifying~$C_k$ with the cyclic subgroup~$\langle\gamma_k\rangle$ of~$T$ generated by the pseudo-rotation of order~$k$ from Example~\ref{ex:pseudo-rot}, the isomorphism~\eqref{eq:colim} follows by cofinality.
\end{proof}



\vfill
\small

\noindent\textsc{Department of Mathematical Sciences, Binghamton University, SUNY}\linebreak
\emph{E-mail address:} \hemail{ross@math.binghamton.edu}\\
\emph{URL:} \hurl[]{people.math.binghamton.edu/ross/}
\medskip

\noindent\textsc{Department of Mathematics and Statistics, University at Albany, SUNY}\linebreak
\emph{E-mail address:} \hemail{mvarisco@albany.edu}\\
\emph{URL:} \hurl{albany.edu/~mv312143/}


\end{document}